\documentclass[11pt,reqno]{amsart}
\usepackage{amsmath,amssymb,geometry,color,tikz-cd}
\usepackage[backref,pagebackref,pdftex,hyperindex]{hyperref}
\usepackage[alphabetic]{amsrefs}
\usepackage{amssymb}
\usepackage{bbm}
\usepackage{enumitem}
\usepackage{upgreek}

\newtheorem{proposition}{Proposition}[section]
\newtheorem{theorem}[proposition]{Theorem}

\newtheorem{conjecture}[proposition]{Conjecture}

\theoremstyle{definition}
\newtheorem{remark}[proposition]{Remark}
\newtheorem{definition}[proposition]{Definition}
\newtheorem{example}[proposition]{Example}
\newtheorem{question}[proposition]{Question}

\title{Approximating delta invariants in the sense of plt complements}
\author{Chuyu Zhou}
\address{\'Ecole Polytechnique F\'ed\'erale de Lausanne (EPFL), MA C3 615, Station 8, 1015 Lausanne, Switzerland}
\email{chuyu.zhou@epfl.ch}
\thanks{2010 
	    \emph{Mathematics Subject Classification}: 14J45.
	    \newline
	    \indent 
		\emph{Keywords}: Fano varieties, K-stability, delta invariants.
	}

\date{} 

\newcommand{\ord}{{\rm {ord}}}
\newcommand{\tc}{{\rm {tc}}}
\newcommand{\vol}{{\rm {vol}}}
\newcommand{\lct}{{\rm {lct}}}

\newcommand{\Proj}{{\rm{Proj}}}

\newcommand{\Val}{{\rm {Val}}}

\newcommand{\dt}{{\rm {dt}}}

\newcommand{\Supp}{{\rm {Supp}}}

\newcommand{\bA}{\mathbb{A}}

\newcommand{\bC}{\mathbb{C}}

\newcommand{\bN}{\mathbb{N}}

\newcommand{\bP}{\mathbb{P}}
\newcommand{\bQ}{\mathbb{Q}}
\newcommand{\bR}{\mathbb{R}}

\newcommand{\bZ}{\mathbb{Z}}

\newcommand{\mD}{\mathcal{D}}
\newcommand{\mE}{\mathcal{E}}
\newcommand{\mF}{\mathcal{F}}

\newcommand{\mH}{\mathcal{H}}

\newcommand{\mL}{\mathcal{L}}

\newcommand{\mO}{\mathcal{O}}

\newcommand{\mX}{\mathcal{X}}
\newcommand{\mY}{\mathcal{Y}}
\newcommand{\mZ}{\mathcal{Z}}

\begin{document}

\begin{abstract}
In this note, we will show that delta invariant  of a log Fano pair can be approximated by lc places of plt complements if it is no greater than one.  Under the assumption that delta invariant (no greater than one) of a log Fano pair can be approximated by lc places of \textit{bounded} plt complements, we show the existence of divisorial valuation computing delta invariant of this log Fano pair.

\end{abstract}

\maketitle
\tableofcontents

\section{Introduction}

Since the establishment of valuative criterion in the works \cite{Fuj19, Li17}, there has been a great progress in the study of K-stability. In particular, the traditional definition of K-stability by test configurations and generalized Futaki invariants can be replaced by beta invariants, and people pay more attention to the divisors over the given Fano variety rather than test configurations. In the work \cite{BHJ17}, 
the relation between test configurations and valuations is well studied, which makes it more clear that valuation theory indeed plays an important role in K-stability. In \cite{Fuj19}, some concepts of good divisors have appeared, such as dreamy divisors and special divisors (see Section \ref{section 2} for the definition), which correspond to good test configurations, and to use beta invariants to test K-stability one only needs to test all special divisors. It is somehow closely related to special test configuration theory developed in \cite{LX14}, which says it is enough to test all special test configurations to check K-stability of a Fano variety. One advantage to look at good divisors is that they can be bounded in some sense, which provides unexpected viewpoints to K-stability. For example, we have known that delta invariants can be approximated by special divisors if they are no greater than one, e.g. \cite{BLZ19}, thus it would be natural to bound these special divisors to study delta invariants. In the work \cite{BLX19}, they show that a weakly special divisor can be achieved as a log canonical place of some bounded complement. Therefore, they can bound these weakly special divisors in the sense of bounded complements, since bounded complements lie in a bounded family. This observation leads to the final solution to the openness conjecture for K-semistable locus and the existence of QM-minimizers for delta invariants, e.g. \cite{BLX19,Xu20}. In this note, we will give a refined approximation result for delta invariants by plt complements. Here a plt complement means a complement admitting a plt crepant pullback.  See Section \ref{section 2} for the definition.

\begin{theorem}\label{mainthm:1}
Given a log Fano pair $(X,\Delta)$ with $\delta(X,\Delta)\leq 1$. Then there is a sequence of prime divisors $\{E_i\}_{i\in \bN}$ over $X$ such that
\begin{enumerate}
\item $\delta(X,\Delta)=\lim_{i\to \infty}\frac{A_{X,\Delta}(E_i)}{S_{X,\Delta}(E_i)}$, 
\item  For each $E_i$ there is a plt complement $(X,\Delta_i^+)$ such that  $E_i$ is the lc place of $(X,\Delta_i^+)$.
\end{enumerate}
\end{theorem}

Note here that we have not confirmed in the above theorem that the plt complement can be replaced by \textit{bounded} plt complement (see Section \ref{section 2} for the definition). This is the following conjecture.

\begin{conjecture}\label{bounded plt}
Let $(X,\Delta)$ be a log Fano pair with $\delta(X,\Delta)\leq 1$, Then there is a natural number $N$ and a sequence of prime divisors $\{E_i\}_{i\in \bN}$ over $X$ such that
\begin{enumerate}
\item $\delta(X,\Delta)=\lim_{i\to \infty}\frac{A_{X,\Delta}(E_i)}{S_{X,\Delta}(E_i)}$, 
\item  For each $E_i$ there is a plt $N$-complement $(X,\Delta_i^+)$ such that  $E_i$ is the lc place of $(X,\Delta_i^+)$.
\end{enumerate}
\end{conjecture}

Once this conjecture is proved, one only needs to consider these bounded plt complements. Obviously they lie in an algebraic family. After we make a stractification of the base of this family, we may assume that for each stratum the restricted family admits a fiberwise log resolution (up to an \'etale morphism). Similar to \cite{BLX19}, but in this case we only look at those strata whose dual complexes of lc places are zero dimensional. We then have the following result.

\begin{theorem}\label{mainthm:2}
Given a log Fano pair $(X,\Delta)$ with $\delta(X,\Delta)\leq 1$. Assume Conjecture \ref{bounded plt} is true, then there exists a prime divisor $E$ over $X$ such that $\delta(X,\Delta)=\frac{A_{X,\Delta}(E)}{S_{X,\Delta}(E)}$.
\end{theorem}

In \cite{BLZ19}, we have shown that delta invariant (no greater than one) of a log Fano pair $(X,\Delta)$ can be approximated by special divisors over $X$. 

\begin{theorem}{\rm {(\cite[Theorem 4.3]{BLZ19})}}\label{special approximation}
Let $(X,\Delta)$ be a log Fano pair with $\delta(X,\Delta)\leq 1$, then there is a sequence of special divisors $\{E_i\}_{i\in \bN}$ such that
$$\delta(X,\Delta)=\lim_{i\to \infty}\frac{A_{X,\Delta}(E_i)}{S_{X,\Delta}(E_i)}. $$
\end{theorem}

However, the approximation by lc places of plt complements in Theorem \ref{mainthm:1} is not finer than the above approximation by special divisors. In fact, there exist lc places of some plt complements which are not  special divisors. 
We will give  a counter-example in the last section. 
The converse direction has not been addressed here, that is, we do not answer the question that whether a special divisor over $X$ can always be achieved as an lc place of some plt complements.

It is now well known to experts that weakly special divisors can be characterized in the sense of complements, e.g. \cite[Appendix]{BLX19}. However, it is not known how to give a similar characterization for special divisors, though this characterization is known in a local viewpoint, that is, there is a correspondence between special divisors over $(X,\Delta)$ and $\bC^*$-equivariant Koll\'ar components over the cone vertex of $(Z,\Gamma)$, where $(Z,\Gamma)$ is the affine cone over $(X,\Delta)$ with respect to a multiple of $-(K_X+\Delta)$, see \cite{LWX21}.  We put here the following question that has not been addressed in this note.

\begin{question}\label{special divisor}
How to give a characterization of special divisors in the sense of complements (just as the characterization for weakly special divisors)?
\end{question}

\begin{remark}
Recently, Ziquan Zhuang has given a positive answer to the converse direction. That is, a special divisor over $X$ can always be achieved as an lc place of some plt complements, thus giving the desired characterization of special divisors. See the survey \cite[Theorem 4.12]{Xu20b}.
\end{remark}

The paper is organized as follows. In Section \ref{section 2} we provide some necessary preliminaries. In Section \ref{section 3}, we prove Theorem \ref{mainthm:1}. In Section \ref{section 4}, we prove Theorem \ref{mainthm:2}. In the last section we provide a counter-example to explain that  lc places of plt complements may not correspond to special divisors.

\noindent
\subsection*{Acknowledgement}
The author thanks Professor Chenyang Xu and Chen Jiang for beneficial comments. Special thanks go to Yuchen Liu for  pointing out a serious mistake and providing the counter-example in the last section.

\section{Preliminaries}\label{section 2}

In this section, we provide some necessary preliminaries. Throughout the note, we always work over the complex number field $\bC$. We say that $(X,\Delta)$ is a log pair if $X$ is a projective normal variety and $\Delta$ is an effective $\bQ$-divisor on $X$ with $K_X+\Delta$ being $\bQ$-Cartier. We say that a log pair $(X,\Delta)$ is log Fano if it admits klt singularities and $-K_X-\Delta$ is ample. We say that $E$ is a prime divisor over $X$ if there is a proper normal birational model $\mu:Y\to X$ such that $E$ is a prime divisor on $Y$. For singularities in birational geometry such as lc, klt, dlt, plt, etc., we refer to \cite{KM98,Kollar13}.

\subsection{Valuations and log discrepancies}

Let $X$ be a variety over $\bC$, a valuation $v$ over $X$ is a real function defined on the function field $K(X)^*:=K(X)\setminus 0$ satisfying the following conditions:
\begin{enumerate}
\item $v(fg)=v(f)+v(g)$,
\item $v(f+g)\geq \min\{v(f), v(g)\},$
\item $v(\bC^*)=0$.
\end{enumerate}
We also set $v(0)=+\infty$. The valuation ring $\mO_v$ associated to $v$ is given by
$$\mO_v:=\{f\in K(X)| v(f)\geq 0\}. $$
We say the valuation $v$ is centered at a scheme-theoretic point $\eta\in X$ if there is a local inclusion $\mO_{X,\eta}\hookrightarrow \mO_v$, denoted by $c_X(v)=\eta$. We write $\Val_X$ for the space of valuations with centers on $X$. Now  we introduce an important class of valuations which are called quasi-monomial valuations. Let $\mu: Y\to X$ be a proper birational morphism and $\eta\in Y$ a scheme-theoretic point on $Y$. Suppose $Y$ is regular at $\eta$ and choose a local system of parameters $y_1,...,y_r\in \mO_{Y,\eta}$ at $\eta$. If $\beta:=(\beta_1,...,\beta_r)\in \bZ_{\geq 0}^r$, we denote by $y^\beta:=\Pi_{j=1}^ry_j^{\beta_j}$. For $f\in \mO_{Y,\eta}$, one can write $f=\sum_{\beta\in \bZ^r_{\geq 0}}c_\beta y^\beta$, with $c_\beta\in \widehat{\mO_{Y,\eta}}$ being zero or a unit. For $\alpha:=(\alpha_1,...,\alpha_r)\in \bR^r_{\geq 0}\setminus 0$, we set
$$v_\alpha(f):=\min_\beta\{(\alpha, \beta):=\sum_{j=1}^r\alpha_j\beta_j| c_\beta\ne 0\}. $$
A quasi-monomial valuation is a valuation of the above form. Now we turn to the definition of  log discrepancies. Let $(X,\Delta)$ be a log pair and $f: Y\to (X,\Delta)$ is a log resolution. For a prime divisor $F\subset Y$, we define 
$$A_{X,\Delta}(F):=\ord_F(K_Y-f^*(K_X+\Delta))+1. $$
Suppose $E_i, i=1,...,r,$ are $r$ simple normal crossing divisors on $Y$ which intersect at a scheme theoretic point $\eta$. We choose a local system of parameters at $\eta$ corresponding to local generators of these $E_i$'s. For the quasi-monomial valuation $v_\alpha$ corresponding to $\alpha=(\alpha_1,...,\alpha_r)$, we define
$$A_{X,\Delta}(v_\alpha):=\sum_{j=1}^r\alpha_jA_{X,\Delta}(E_j).$$

\subsection{Dual complex of lc places}

Let $(Y,E=\sum_j E_j)$ be a simple normal crossing pair, one can form a dual complex $\mD(E)$ as follows. For each component $E_j$ we assign a vertex $v_j$, and for any component $Z$ of the intersection of $r$ components $E_{j_k}, k=1,...,r,$ we assign the following $(r-1)$-cell
$$W_Z:=\{(a_1,...,a_r)\in \bR^r_{\geq 0}| \sum_{k=1}^{r}a_k=1\} $$
glued on $v_{j_1},..., v_{j_r}$. It is not hard to see that the points in $W_Z$ correspond to quasi-monomial valuations over $Y$.

Let $(X,\Delta)$ be a log canonical pair and $\mu: Y\to (X,\Delta)$ be a log resolution. Suppose $E_i, i=1,...,r,$ are all simple normal crossing lc places on $Y$. Write $E:=\sum_{j=1}^r E_j$ the reduced sum of these lc places, then $\mD(E)$ is called the dual complex of lc places of the pair $(X,\Delta)$ with respect to the log resolution $\mu$. If $(X,\Delta)$ admits a plt crepant pullback (i.e. there exists a proper birational morphism from a log pair $(Y,\Delta_Y)\stackrel{f}{\to} (X,\Delta)$ such that $K_Y+\Delta_Y=f^*(K_X+\Delta)$ and $(Y,\Delta_Y)$ is plt), then for any log resolution, it is clear that the dual complex of lc places is zero dimensional.

\subsection{Good test configurations and good divisors}

\begin{definition}
Let $(X,\Delta)$ be a log Fano pair. A test configuration $(\mX,\Delta_{\tc};\mL)\to \bA^1$ consists of the following data:
\begin{enumerate}
\item A flat morphism $(\mX,\Delta_{\tc})\to \bA^1$ from a normal scheme.
\item A $\bC^*$-equivariant isomorphism between $(\mX,\Delta)\times_{\bA^1}\bC^*$ and $(X\times \bA^1, \Delta\times \bA^1)\times_{\bA^1}\bC^*$ induced by the natural $\bC^*$-action on $\bA^1$.
\item A relatively ample $\bQ$-line bundle $\mL$ on $\mX$ which is $\bC^*$-equivariant and $\mL_t\sim_\bQ -K_X-\Delta$ for $t\neq 0$.
\end{enumerate}
We say that the test configuration is dreamy if the central fiber $\mX_0$ is integral. We say that the test configuration is weakly special (resp. special) if it is dreamy and $(\mX,\Delta_{\tc}+\mX_0)$ is log canonical (resp. plt), and $\mL\sim_\bQ -K_{\mX}-\Delta_{\tc}$.  Note that a special test configuration is equivalent to a log Fano degeneration of $(X,\Delta)$ since the central fiber $(\mX_0,\Delta_{\tc,0})$ is also a log Fano pair.
\end{definition}

For a dreamy test configuration $(\mX,\Delta_{\tc};\mL)$, the central fiber induces a divisorial valuation $\ord_{\mX_0}$ over $\mX$. If one restricts the divisorial valuation via the embedding $K(X)^*\hookrightarrow K(\mX^*)$, then we get a divisorial valuation over $X$ if the test configuration is non-trivial, e.g. \cite[Lemma 4.1]{BHJ17}. We denote by ${\ord_{\mX_0}}|_X$ the restriction of $\ord_{\mX_0}$, then there exists a natural number $c$ and a prime divisor $E$ over $X$ such that ${\ord_{\mX_0}}|_{X}=c\cdot \ord_E$. Note that we can assume $c=1$ up to a base change of $(\mX,\Delta_{\tc}; \mL)\to \bA^1$.

\begin{definition}
If $(\mX,\Delta_{\tc};\mL)\to \bA^1$ is a dreamy test configuration, we say that the corresponding $E$ (resp. $\ord_E$) over $X$ is a dreamy divisor (resp. dreamy valuation). If the test configuration is weakly special, we say that $E$ (resp. $\ord_E$) over $X$ is a weakly special divisor (resp. weakly special valuation). If the test configuration is special, we say that $E$ (resp. $\ord_E$) over $X$ is a special divisor (resp. special valuation). 
\end{definition}

\subsection{Complements}

The concept of complement is originally introduced by V.V.Shokurov in the work \cite{Sho92}, and it is systematically developed by Birkar in \cite{Birkar19}. It is proved that this concept indeed plays an important role in the study of singularities of anti-canonical linear systems of log Fano pairs and the final solution to BAB conjecture, e.g. \cite{Birkar19, Birkar21}.

\begin{definition}
Let $(X,\Delta)$ be a log Fano pair. An effective $\bQ$-divisor $\Delta^+$ (or a pair $(X,\Delta^+)$) is called a complement of $(X,\Delta)$ if the pair $(X,\Delta^+)$ is log canonical with $K_X+\Delta^+\sim_\bQ 0$ and $\Delta^+\geq \Delta$. We say that $\Delta^+$ (or a pair $(X,\Delta^+)$) is a $N$-complement for some positive natural number $N$ if moreover $N(K_X+\Delta^+)\sim 0$ and $N\Delta^+\geq N\lfloor\Delta\rfloor +\lfloor(N+1)\{\Delta\} \rfloor$, where $\{\Delta\}=\Delta-\lfloor\Delta\rfloor$.
\end{definition}

A key observation in \cite{BLX19} is that weakly special divisors can be bounded in the sense of bounded complements.

\begin{theorem}{\rm (\cite[Appendix]{BLX19})}
Let $(X,\Delta)$ be a log Fano pair and $E$ a weakly special divisor over $X$. Then there exists a $N$-complement $(X,\Delta^+)$ such that $E$ can be achieved as an lc place of $(X,\Delta^+)$, where $N$ only depends on the pair $(X,\Delta)$ but not depends on $E$.
\end{theorem}

One can choose $N$ sufficiently divisible such that $-N(K_X+\Delta)$ is Cartier and the complement $\Delta^+$ lies in the linear system $\frac{1}{N}|-N(K_X+\Delta)|$. Thus it is clear that $(X,\Delta^+)$ lies in a bounded family and any weakly special divisor over $X$ can be achieved as an lc place of some fiber of the family.

In this note, we concern the following refined concepts of complements, i.e. plt complements.

\begin{definition}
Let $(X,\Delta)$ be a log Fano pair. The complement $(X,\Delta^+)$ is called a plt complement if there is a proper birational morphism from a log pair $(Y,\Delta_Y)$, denoted by $f: (Y, \Delta_Y)\to (X,\Delta)$, such that $K_Y+\Delta_Y= f^*(K_X+\Delta^+)$ and $(Y,\Delta_Y)$ admits plt singularities.
A plt complement $(X,\Delta^+)$ is called a plt $N$-complement for some positive natural number $N$ if $N(K_X+\Delta^+)\sim 0$.
\end{definition}

\subsection{Delta invariants}

Let $(X,\Delta)$ be a log Fano pair. For a divisible natural number $m$, a $m$-basis type divisor is of the following form 
$$D_m:=\frac{\sum_i {\rm div}\{s_i=0\}}{m\dim H^0(X,-m(K_X+\Delta))}\sim_\bQ -K_X ,$$
where $\{s_i\}_i$ is a complete basis of the vector space $H^0(X,-m(K_X+\Delta))$.
Then we can define the following $m$-th delta invariant of $(X,\Delta)$ due to \cite{FO18},
$$\delta_m(X,\Delta):=\inf_{D_m} \lct(X,\Delta; D_m) ,$$
where $D_m$ runs over all $m$-basis type divisors. We also note here that the infimum can be achieved by some $m$-basis type divisor, that is,  $\inf=\min$ in this case.

By \cite{FO18, BJ20}, the limit $\lim_{m\to \infty}\delta_m(X,\Delta)$ indeed exists which is exactly the now known delta invariant, i.e. $\delta(X,\Delta)$. For a prime divisor $E$ over $X$, we define
$$S_m(E):=\sup_{D_m}\ord_E(D_m), $$
where $D_m$ runs over all $m$-basis type divisors, then $\delta_m(X,\Delta)=\frac{A_{X,\Delta}(E)}{S_m(E)}$. We also know that
$$S_{X,\Delta}(E):=\lim_{m\to \infty} S_m(E)=\frac{1}{\vol(-K_X-\Delta)}\int_0^\infty\vol(-\mu^*(K_X+\Delta)-tE)\dt .$$
Then we have the following result on the definition of delta invariants due to \cite{BJ20}.

\begin{theorem}
Let $(X,\Delta)$ be a log Fano pair, then
$$\delta(X,\Delta)=\inf_E\frac{A_{X,\Delta}(E)}{S_{X,\Delta}(E)}, $$
where $E$ runs over all prime divisors over $X$.
\end{theorem}
Delta invariant is a K-stability threshold due to the following well-known theorem.

\begin{theorem}{\rm {(\cite{Fuj19, FO18,BJ20})}}
Let $(X,\Delta)$ be a log Fano pair, then 
\begin{enumerate}
\item $(X,\Delta)$ is K-semistable if and only if $\delta(X,\Delta)\geq 1$.
\item $(X,\Delta)$ is uniformly K-stable if and only if $\delta(X,\Delta)>1$.
\end{enumerate}
\end{theorem}

\section{On plt complements}\label{section 3}
In this section, we prove Theorem \ref{mainthm:1}.

\begin{theorem}{\rm (= Theorem \ref{mainthm:1})}
Given a log Fano pair $(X,\Delta)$ with $\delta(X,\Delta)\leq 1$. Then there exists a sequence of prime divisors $\{E_i\}_{i\in \bN}$ over $X$ such that
\begin{enumerate}
\item $\delta(X,\Delta)=\lim_{i\to \infty}\frac{A_{X,\Delta}(E_i)}{S_{X,\Delta}(E_i)}$, 
\item  For each $E_i$ there is a plt complement $(X,\Delta_i^+)$ such that  $E_i$ is the lc place of $(X,\Delta_i^+)$.
\end{enumerate}
\end{theorem}

\begin{proof}

We first assume $\delta(X,\Delta)<1$, then the $m$-th delta invariant $\delta_m$ is smaller than one for each sufficiently divisible natural number $m$. For each such $m$, there exists a $m$-basis type divisor $D_m$ of $(X,\Delta;-K_X-\Delta)$ such that $(X,\Delta+\delta_mD_m)$ is strictly log canonical (i.e. lc but not klt). Take a log resolution of  the pair $f: Y\to (X,\Delta+\delta_mD_m)$, we write
$$K_Y+f_*^{-1}\Delta+\delta_m f_*^{-1}D_m=f^*(K_X+\Delta+\delta_mD_m)+\sum_j a_jF_j-\sum_k G_k ,$$
where 
\begin{enumerate}
\item $(Y, f_*^{-1}\Delta+ f_*^{-1}D_m+\sum_j F_j+\sum_k G_k)$ is simple normal crossing, 
\item $F_j,G_k$ are all exceptional divisors with $a_j>-1$,
\item $G_k$ are all exceptional lc places on $Y$.
\end{enumerate}
We first choose positive rational numbers $r_j, l_k$ such that $-(\sum_jr_j F_j+\sum_k l_k G_k)$ is $f$-ample. Then we choose a rational number $0<\epsilon_m\ll 1$ and write the following formula:
$$K_Y+f_*^{-1}\Delta+(1-\epsilon_m)\delta_m f_*^{-1}D_m=f^*(K_X+\Delta+(1-\epsilon_m)\delta_mD_m)+\sum_j a_j'F_j-\sum_k(1-b_k) G_k , $$
where $a_j'>a_j>-1$ and $0<b_k \ll1$. More precisely, we have 
$$a_j'-a_j=\ord_{F_j}(\epsilon_m \delta_mD_m) \quad and \quad b_k=\ord_{G_k}(\epsilon_m\delta_m D_m).$$
We next choose an ample $\bQ$-divisor $H\sim_\bQ -K_X-\Delta$ such that the pair
$$(X,\Delta+(1-\epsilon_m)\delta_mD_m+\{1-(1-\epsilon_m)\delta_m\}H)$$ 
is klt. One can always choose $l_k$ properly such that the set $\{ \frac{b_k}{l_k}\}_k$ has a unique minimal element, and we just assume $\min_k \{ \frac{b_k}{l_k}\}_k=\frac{b_1}{l_1}$ for convenience. Consider $f: Y\to (X,\Delta+(1-\epsilon_m)\delta_mD_m+\{1-(1-\epsilon_m)\delta_m\}H)$ and write

\begin{align*}
&K_Y+f_*^{-1}\Delta+(1-\epsilon_m)\delta_m f_*^{-1}D_m+\{1-(1-\epsilon_m)\delta_m\}f^*H-\frac{b_1}{l_1}(\sum_j r_jF_j+\sum_k l_k G_k)\\
=&f^*(K_X+\Delta+(1-\epsilon_m)\delta_mD_m+\{1-(1-\epsilon_m)\delta_m\}H)\\&+\sum_j a_j'F_j-\sum_k(1-b_k) G_k -\frac{b_1}{l_1}(\sum_j r_jF_j+\sum_k l_k G_k)\\
=& f^*(K_X+\Delta+(1-\epsilon_m)\delta_mD_m+\{1-(1-\epsilon_m)\delta_m\}H)\\&+\sum_j (a_j'-\frac{b_1r_j}{l_1})F_j-G_1-\sum_{k\geq 2}(1-b_k+\frac{b_1l_k}{l_1})G_k
\end{align*}
It is clear that $1-b_k+\frac{b_1l_k}{l_1}<1$ for $k\geq 2$ by our assumption that the set $\{ \frac{b_k}{l_k}\}_k$ has a unique minimal element $\frac{b_1}{l_1}$. We denote by
$$L:= \{1-(1-\epsilon_m)\delta_m\}f^*H-\frac{b_1}{l_1}(\sum_j r_jF_j+\sum_k l_k G_k) .$$
One can always choose $\epsilon_m$ sufficiently small to make $L$ ample and $a_j'-\frac{b_1r_j}{l_1}>-1$ (recall that $b_k=\ord_{G_k}(\epsilon_m\delta_mD_m)$, thus $\epsilon_m$ being sufficiently small makes sure $b_1$ being sufficiently small). We then have the following 

\begin{align*}
&K_Y+f_*^{-1}\Delta+(1-\epsilon_m)\delta_m f_*^{-1}D_m+L+\sum_j F_j+\sum_k G_k\\
=& f^*(K_X+\Delta+(1-\epsilon_m)\delta_mD_m+\{1-(1-\epsilon_m)\delta_m\}H)\\ +& \sum_j (1+a_j'-\frac{b_1r_j}{l_1})F_j+\sum_{k\geq 2} (b_k-\frac{b_1l_k}{l_1} )G_k,
\end{align*}
where $1+a_j'-\frac{b_1r_j}{l_1}>0$ for any $j$ and $b_k-\frac{b_1l_k}{l_1}>0$ for any $k\geq 2$. We find an ample $\bQ$-divisor $\tilde{L}\sim_\bQ L$ such that
\begin{enumerate}
\item $\Supp(\tilde{L})$ does not contain $G_1$,
\item the pair $(Y, f_*^{-1}\Delta+(1-\epsilon_m)\delta_m f_*^{-1}D_m+\tilde{L}+\sum_j F_j+\sum_k G_k)$ is dlt.
\end{enumerate}
 By \cite[Corollary 1.4.2]{BCHM10}, one can run MMP/$X$ for the pair $(Y, f_*^{-1}\Delta+(1-\epsilon_m)\delta_m f_*^{-1}D_m+\tilde{L}+\sum_j F_j+\sum_k G_k)$
to get a minimal model $\phi: Y\dashrightarrow Y'/X$ such that $\phi$ contracts all $\{ F_j\}_j$ and $\{ G_k\}_{k\geq 2}$, and $\phi$ only contracts these divisors. We denote by $\Delta', D_m', G_1', L'$ the push forward of $f_*^{-1}\Delta, f_*^{-1}D_m, G_1,\tilde{L}$ by $\phi$ respectively, then we get a morphism 
$$f': (Y', \Delta'+(1-\epsilon_m)\delta_mD_m'+L'+G_1') \to (X, \Delta+(1-\epsilon_m)\delta_mD_m+f'_*L').$$
It is clear that 
\begin{enumerate}
\item $K_X+\Delta+(1-\epsilon_m)\delta_mD_m+f'_*L'\sim_\bQ 0,$
\item $K_{Y'}+\Delta'+(1-\epsilon_m)\delta_mD_m'+L'+G_1'=f'^*(K_X+\Delta+(1-\epsilon_m)\delta_mD_m+f'_*L'), $
\item the pair $(Y', \Delta'+(1-\epsilon_m)\delta_mD_m'+L'+G_1')$ is plt.
\end{enumerate}
Denote by $\Delta_m^+:=\Delta+ (1-\epsilon_m)\delta_mD_m+f'_*L'$ and $E_m:=G_1'$, then we see $(X,\Delta_m^+)$ is a plt complement and $E_m$ is the lc place of $(X,\Delta_m^+)$. By our choice, we have $\delta_m=\frac{A_{X,\Delta}(E_m)}{\ord_{E_m}(D_m)}$. Thus by \cite[Theorem 4.1]{BLZ19}, we see that $\{E_m \}$ is a sequence of prime divisors over $X$ we need to approximate $\delta(X,\Delta)$.

Now we turn to the case $\delta(X,\Delta)=1$. By \cite[Theorem 3.1]{ZZ19}, for any rational number $0<\epsilon \ll 1$ one can find a $0\leq B\sim_\bQ -K_X-\Delta$ such that $\delta(X, \Delta+\epsilon B)<1$. We choose a decreasing sequence of rational numbers $0<\epsilon_i\ll1$ such that $\lim_{i\to \infty} \epsilon_i=0$. For each $\epsilon_i$ we choose a $0\leq B_i\sim_\bQ -K_X-\Delta$ such that $\delta(X,\Delta+\epsilon_iB_i)<1$. By the previous case, for each $\epsilon_i$, one can find a sequence of prime divisors $\{E_{i,m}\}_{m\in \bN}$ which can be achieved as lc places of some plt complements of $(X,\Delta+\epsilon_i B_i)$ (hence also plt complements of $(X,\Delta)$) such that

$$\delta(X, \Delta+\epsilon_iB_i)=\lim_{m\to \infty} \frac{A_{X,\Delta+\epsilon_iB_i}(E_{i,m})}{S_{X,\Delta+\epsilon_iB_i}(E_{i,m})}.$$

Note that

$$\frac{A_{X,\Delta+\epsilon_iB_i}(E_{i,m})}{S_{X,\Delta+\epsilon_iB_i}(E_{i,m})}= \frac{A_{X,\Delta}(E_{i,m})-\ord_{E_{i,m}}(\epsilon_iB_i)}{(1-\epsilon_i)S_{X,\Delta}(E_{i,m})}
\geq \frac{A_{X,\Delta}(E_{i,m})-\epsilon_i T_{X,\Delta}(E_{i,m})}{(1-\epsilon_i)S_{X,\Delta}(E_{i,m})},$$
where 
$$T_{X,\Delta}(E_{i,m}):=\sup_{D\in |-K_X-\Delta|_\bQ} \ord_{E_{i,m}}(D).$$
That is,
$$ \frac{A_{X,\Delta}(E_{i,m})}{S_{X,\Delta}(E_{i,m})}\leq (1-\epsilon_i)\cdot\frac{A_{X,\Delta+\epsilon_iB_i}(E_{i,m})}{S_{X,\Delta+\epsilon_iB_i}(E_{i,m})}+\epsilon_i\cdot \frac{T_{X,\Delta}(E_{i,m})}{S_{X,\Delta}(E_{i,m})}.$$
As $\frac{T_{X,\Delta}(E_{i,m})}{S_{X,\Delta}(E_{i,m})}\leq n+1$ (see \cite[Section 3]{BJ20}), where $n$ is the dimension of $X$, thus 
$$\inf_{i\to \infty}\inf_{m\to\infty} \frac{A_{X,\Delta}(E_{i,m})}{S_{X,\Delta}(E_{i,m})}\leq 1 .$$
So one can find a subsequence of $\{E_{i,m}\}_{i,m\in \bN}$ to approximate $\delta(X,\Delta)$. The proof is finished.

\end{proof}

\section{On optimal destabilization conjecture}\label{section 4}

For a log Fano pair $(X,\Delta)$ with $\delta(X,\Delta)\leq 1$, the existence of divisorial valuation computing $\delta(X,\Delta)$ is one of the core problems in algebraic K-stability theory, e.g. \cite[Conjecture 1.5]{BX19} or \cite[Conjecture 1.2]{BLZ19}. In this section we will prove Theorem \ref{mainthm:2} which is a tentative attempt to this problem.

\begin{theorem}{\rm{(=Theorem \ref{mainthm:2})}}
Given a log Fano pair $(X,\Delta)$ with $\delta(X,\Delta)\leq 1$. Assume Conjecture \ref{bounded plt} is true, then there exists a prime divisor $E$ over $X$ such that $\delta(X,\Delta)=\frac{A_{X,\Delta}(E)}{S_{X,\Delta}(E)}$.
\end{theorem}

\begin{proof}

We assume that there exists a sufficiently divisible natural number $N$ such that $\delta(X,\Delta)$ can be approximated by using lc places of plt $N$-complements. Denote by $\bP^M:=\frac{1}{N}|-N(K_X+\Delta)|$ where $M+1$ is the dimension of the linear system $|-N(K_X+\Delta)|$, and let $\mH\subset X\times \bP^M$ be the universal divisor with respect to $\frac{1}{N}|-N(K_X+\Delta)|$. We consider the family 
$$\pi: (X\times \bP^M, \Delta\times \bP^M+\mH)\to \bP^M .$$
By ACC of log canonical thresholds (e.g. \cite{HMX14}) and lower semi-continuity of log canonical thresholds, there exists a locally closed subset $\mZ\subset \bP^M$ such that the corresponding restricted family $\pi_\mZ: (X\times \mZ, \Delta\times \mZ+\mH_\mZ)\to \mZ$ satisfies the following properties:
$$\mZ=\{t\in \bP^M|\textit{the fiber $(X, \Delta+\mH_t)$ is strictly log canonical, i.e. lc but not klt}\}. $$
By a similar way as the proof of \cite[Proposition 4.3 and Theorem 4.5]{BLX19}, there exists a finite decomposition $\mZ=\bigcup_{i\in I}\mZ_i$ so that each $\mZ_i$ is smooth and the restricted family $\pi_{\mZ_i}: (X\times \mZ_i, \Delta\times \mZ_i+\mH_{\mZ_i})\to \mZ_i$ admits a fiberwise log resolution up to an \'etale morphism. That is to say, for each $i$ there exists an \'etale morphism $\mZ_i'\to \mZ_i$ such that
the restricted family (via a base change) $\pi_{\mZ_i'}: (X\times \mZ_i', \Delta\times \mZ_i'+\mH_{\mZ_i'})\to \mZ_i'$ admits a fiberwise log resolution. As it will not cause any confusion we still use $\mZ_i$ to replace $\mZ_i'$ for convenience. We take a fiberwise log resolution of $(X\times \mZ_i, \Delta\times \mZ_i+\mH_{\mZ_i})$ for each $i$, denoted by
$$f_i: \mY_i\to (X\times \mZ_i, \Delta\times \mZ_i+\mH_{\mZ_i}). $$
Then for each closed point $t\in \mZ_i$, the fiber $f_{i,t}: \mY_{i,t}\to (X, \Delta+\mH_t)$ is a log resolution. As the pair $(X\times \mZ_i, \Delta\times \mZ_i+\mH_{\mZ_i})$ is strictly log canonical, we use $\mE_i$ to denote the reduced sum of all lc places on $\mY_i$. Then one sees that $\mE_i$ generates a dual complex of lc places over $(X\times \mZ_i,\Delta\times \mZ_i+\mH_{\mZ_i})$, denoted as $\mD(\mE_i)$. Due to the fiberwise log resolution, it is also clear that the restriction of $\mD(\mE_i)$ to a fiber over a closed point $t\in \mZ_i$, denoted by $\mD(\mE_{i,t})$, is exactly the dual complex of lc places over $(X,\Delta+\mH_t)$. For a fixed closed point $t_0\in \mZ_i$, we define 
$$a_i:=\inf_{v\in \mD(\mE_{i,t_0})}\frac{A_{X,\Delta}(v)}{S_{X,\Delta}(v)},$$
where $v$ runs through all points in $\mD(\mE_{i,t_0})$ which correspond to quasi-monomial valuations. By \cite[Proposition 4.2]{BLX19}, the number $a_i$ does not depend on the choice of $t_0\in \mZ_i$. By our assumption, 
$$\delta(X,\Delta)=\min_{i\in I'} \{a_i\}_i,$$
where $I'\subset I$ consists of all subscripts $i$ such that $\mD(\mE_i)$ is zero-dimensional. Note that for each $i\in I'$, the number $a_i$ is achieved by a divisorial valuation over $(X, \Delta+\mH_{t_0})$, since the dual complex $\mD(\mE_i)$ (resp. $\mD(\mE_{i,t_0})$) is zero-dimensional. Thus $\delta(X,\Delta)$ is achieved by a divisorial valuation as the index set $I'$ is finite. The proof is finished.

\end{proof}

\section{Example}

In this section, we give a counter-example to explain that  lc places of complements of plt type  are not necessarily  related to special divisors.

\begin{example}
Let $X:=\bP^2$ and $\Delta^+:=C$, where $C\subset \bP^2$ is a smooth cubic curve. Then clearly $(X,\Delta^+)$ is a complement of plt type. However, $C$ is not a special divisor. We also point out here that $C$ is only a weakly special divisor.
\end{example}

We give some details to explain this example. Denote by $R:=\oplus_{m\in \bN} R_m$, where $R_m:=H^0(X,-mK_X)$. Then the divisorial valuation $\ord_C$ naturally induces a filtration $\mF_{\ord_C}$ on the graded ring $R$ as follows:
$$\mF_{\ord_C}^j R_m:=H^0(X, -mK_X-jC)=H^0(\bP^2, \mO_{\bP^2}(3(m-j))). $$
Write $\mX:=\Proj \oplus_{m\in \bN} \oplus_{j\in \bN}t^{-j}\mF^j_{\ord_C}R_m$. By \cite{BHJ17} we know that $\mX\to \bA^1$ is a test configuration induced by $\ord_C$. To show it is not a special test configuration, it suffices to confirm that the central fiber is not klt. The central fiber can be formulated as follows:
$$\mX_0=\Proj \oplus_{m\in \bN}\oplus_{j\in \bN} t^{-j}\mF^j_{\ord_C}R_m/\mF^{j+1}_{\ord_C}R_m\cong \Proj \oplus_{m\in \bN}\oplus_{j\in \bN}t^{-j}H^0(C, \mO_C(3(m-j))).$$
This is a cone over the cubic curve $C$, which is strictly log canonical.

\begin{remark}
This example is provided by Yuchen Liu when the author was considering Question \ref{special divisor} to characterize special divisors in the sense of plt complements.
\end{remark}

\bibliography{reference.bib}
\end{document}